\documentclass[11pt]{amsart}

\usepackage{amsmath}
\usepackage{fullpage}
\usepackage{xspace}
\usepackage[psamsfonts]{amssymb}
\usepackage[latin1]{inputenc}
\usepackage{graphicx,color}
\usepackage[curve]{xypic} 
\usepackage{hyperref}
\usepackage{graphicx}
\usepackage{amsmath}%
\usepackage{amsthm}%
\usepackage{amscd}
\usepackage{amsfonts}%
\usepackage{amssymb}%
\usepackage{graphicx}

\usepackage{mathrsfs}

\usepackage{tikz}
\usetikzlibrary{matrix,arrows}

\newtheorem{theorem}{Theorem}[section]

\newtheorem{conjecture}[theorem]{Conjecture}
\newtheorem{corollary}[theorem]{Corollary}

\newtheorem{lemma}[theorem]{Lemma}

\theoremstyle{remark}

\numberwithin{equation}{section}

\newcommand{\Acal}{\mathscr{A}}

\newcommand{\Kcal}{\mathscr{K}}
\newcommand{\Lcal}{\mathscr{L}}

\newcommand{\Ocal}{\mathscr{O}}
\newcommand{\Pcal}{\mathscr{P}}

\newcommand{\Pro}{\mathbb{P}}

\newcommand{\C}{\mathbb{C}}

\newcommand{\Q}{\mathbb{Q}}
\newcommand{\R}{\mathbb{R}}

\newcommand{\rad}{\mathrm{rad}}

\newcommand{\Pic}{\mathrm{Pic}}

\newcommand{\rk}{\mathrm{rank}\,}
\newcommand{\supp}{\mathrm{supp}}

\newcommand{\NS}{\mathrm{NS}}
\newcommand{\Num}{\mathrm{Num}}

  \DeclareFontFamily{U}{wncy}{}
    \DeclareFontShape{U}{wncy}{m}{n}{<->wncyr10}{}
    \DeclareSymbolFont{mcy}{U}{wncy}{m}{n}
    \DeclareMathSymbol{\Sha}{\mathord}{mcy}{"58}

\begin{document}
\title[]{On the arithmetic case of Vojta's conjecture with truncated counting functions}

\author{Hector Pasten}
\address{ Departamento de Matem\'aticas\newline
\indent Pontificia Universidad Cat\'olica de Chile\newline
\indent Facultad de Matem\'aticas\newline
\indent 4860 Av. Vicu\~na Mackenna\newline
\indent Macul, RM, Chile}
\email[H. Pasten]{hector.pasten@mat.uc.cl}%

%\thanks{}
\thanks{Supported by ANID (ex CONICYT) FONDECYT Regular grant 1190442 from Chile.}
\date{\today}
%\dedicatory{}

\begin{abstract}  We prove a Diophantine approximation inequality for rational points in varieties of any dimension, in the direction of Vojta's conjecture with truncated counting functions. Our results also provide a bound towards the $abc$ conjecture which in several cases is subexponential. The main theorem gives a lower bound for the truncated counting function relative to a divisor with sufficiently many components, in terms of the proximity to an algebraic point. Furthermore, we show that the Lang-Waldschmidt conjecture implies a special case of Vojta's conjecture with truncation in arbitrary dimension. Our methods are based on the theory of linear forms in logarithms and a geometric construction.
\end{abstract}

\maketitle

%\tableofcontents

%%%%%%%%%%%%%%%%%%%%%%%%%%%%%%%%%%%%%%
%%%%%%%%%%%%%%%%%%%%%%%%%%%%%%%%%%%%%%
%%%%%%%%%%%%%%%%%%%%%%%%%%%%%%%%%%%%%%
%%%%%%%%%%%%%%%%%%%%%%%%%%%%%%%%%%%%%%
%%%%%%%%%%%%%%%%%%%%%%%%%%%%%%%%%%%%%%
%%%%%%%%%%%%%%%%%%%%%%%%%%%%%%%%%%%%%%

\section{Introduction}

%%%%
%%%%
%%%%
\subsection{Diophantine approximation with truncated counting functions} In \cite{VojtaABC}, Vojta proposed a far-reaching generalization of the $abc$ conjecture. Although  his conjecture was originally formulated  for rational and algebraic points, it also admits analogues in the setting of  holomorphic maps (Nevanlinna theory) and function fields, see \cite{VojtaCIME}. In the arithmetic case, Vojta's conjecture is a Diophantine approximation statement in varieties of any dimension and involving truncated counting functions; the latter are a generalization of (the logarithm of) the radical of an integer. 

While there are several unconditional results towards Vojta's conjecture in the setting of Nevanlinna theory as well as for function fields, at this point the case of rational points remains largely mysterious. The only unconditional results towards the arithmetic case of Vojta's conjecture with truncated counting functions are proved in dimension one, namely, exponential bounds towards the $abc$ conjecture \cite{ABC1,ABC2, ABC3, MurtyPasten}. 

In this work we provide the first unconditional result towards Vojta's conjecture with truncated counting functions in varieties of arbitrary dimension. Our main result is Theorem \ref{ThmMain2} (with the companion Theorem \ref{ThmRat}) which bounds from below the truncated counting function relative to a divisor with sufficiently many components, in terms of the proximity function to an algebraic point. For the sake of comparison, we will also formulate Corollary \ref{CoroMain} which is more reminiscent of Vojta's conjecture, although it should be remarked that Theorems \ref{ThmMain2} and  \ref{ThmRat} give a stronger result. In general terms, Corollary \ref{CoroMain} asserts that given a smooth projective variety $X$ over $\Q$, a reduced effective divisor $D$ on $X$ with sufficiently many components over $\Q$, a place $v$, an algebraic point $P$ in $X-\supp(D)$, and $\epsilon>0$, we have
$$
\lambda_{X,v}(P,x) < \exp\left(\epsilon\cdot  N_X^{(1)}(D,x) \right) +(\log^* h_X(\Ocal(D),x))^{1+\epsilon} + O(1)
$$
for all rational points $x\in X(\Q)$ outside certain proper Zariski closed subset of $X$. Here, $\lambda_{X,v}(P,x)$ denotes the $v$-adic proximity of $x$ to $P$, $N_X^{(1)}(D,x)$ is the truncated counting function relative to $D$, $h_X(\Ocal(D),x)$ is the height relative to the sheaf $\Ocal(D)$, and $\log^*t=\log \max\{e, t\}$ ---a more detailed description of the notation appears later in this introduction.

We begin our discussion with Theorem \ref{ThmMain1} which is an effective Diophantine approximation result for algebraic numbers, involving the radical of an integer. Theorem \ref{ThmMain1} is relevant in our study since it is a key input to prove our main results in higher dimensional varieties. After discussing an application of Theorem \ref{ThmMain1} in the context of the $abc$ conjecture providing a bound which in several cases is subexponential (see Corollary \ref{CoroSubExpABC}), we present the necessary background to formulate Vojta's conjecture. With these preliminaries we will be able to introduce our main results, namely, Theorems \ref{ThmMain2} and  \ref{ThmRat} along with Corollary \ref{CoroMain}.   

Finally, we will include a discussion on a conjectural improvement: We will show that the Lang-Waldschmidt conjecture from transcendence theory implies a special case of Vojta's conjecture with truncated counting functions in arbitrary dimension. See Theorem \ref{ThmLW}.

It should be noted that although Theorem \ref{ThmMain1} is used as a tool to prove our main results, it can be of independent interest due to the fact that it is effective and because of its Corollary \ref{CoroSubExpABC}.

%%%%
%%%%
%%%%
\subsection{Diophantine approximation of algebraic numbers} Recall that $\log^* t = \log \max\{e, t\}$ and let us write $\log^*_r(t)$ for the $r$-th iteration of $\log^*$. For $x=b/c\in \Q$ with $b,c$ coprime integers, the (logarithmic) height is $h(x)=\log \max\{|b|,|c|\}$. For a non-zero integer $m$ we let $\rad(m)$ be the radical of $m$, that is, the product of the different positive primes dividing $m$ without repetition. We write $M_\Q=\{\infty,2,3,5,...\}$ for the set of places of $\Q$ and for each $v\in M_\Q$ we fix an absolute value $|-|_v$ on $\overline{\Q_v}$ extending the usual $v$-adic one in $\Q$. We also fix an embedding $\overline{\Q}\subseteq \overline{\Q_v}$.

The following (effective)  Diophantine approximation estimate for algebraic numbers will be a crucial tool in our main results.
\begin{theorem} \label{ThmMain1} Let $v\in M_\Q$, let $\alpha\in \overline{\Q}^\times$ and let $\epsilon>0$. There is a number $\kappa(v,\alpha, \epsilon)$ depending effectively on $v$, $\alpha$, and $\epsilon$ such that for all $x=b/c\in \Q^\times -\{\alpha\}$ with $b,c$ coprime integers, we have
$$
- \log |\alpha-x|_v < (\log^* h(x))\exp\left(\frac{(1+\epsilon)\log^* \rad(bc)}{\log^*_2 \rad(bc)}\log^*_3 \rad(bc) + \kappa(v,\alpha,\epsilon)\right).
$$
\end{theorem}

This result is deduced from the theory of linear forms in logarithms. In the archimedian case one can use a celebrated result of Matveev \cite{Matveev} while in the $p$-adic case one can use the non-archimedian counterpart due to Yu \cite{Yu}. Actually, rather than directly applying the results of Matveev and Yu (which is a possible approach) we will use a theorem due to Evertse and Gy\"ory (Theorem 4.2.1 in \cite{EGbook}) concerning effective approximations of algebraic numbers by finitely generated subgroups of the multiplicative group of a number field; this result simplifies the argument in the non-archimedian case. The Evertse-Gy\"ory theorem is proved using the aforementioned results of Matveev and Yu, along with some new results from geometry of numbers.

%%%%
%%%%
%%%%
\subsection{Application: Subexponential $abc$} Let us recall the $abc$ conjecture of Masser and Oesterl\'e.

\begin{conjecture}[$abc$ conjecture] Let $\epsilon>0$. There is a constant $K_\epsilon>0$ such that the following holds:  Let $a$, $b$, $c$ be coprime positive integers with $a+b=c$. Then $c< K_\epsilon \cdot \rad(abc)^{1+\epsilon}$.
\end{conjecture}
Consequences of this conjecture are well-known and discussed elsewhere. Instead, here we will focus on unconditional partial progress on the problem. 

Under the assumptions of the $abc$ conjecture, let us write $R=\rad(abc)$. Currently, the best unconditional bound in the direction of this conjecture is of exponential nature, proved by Stewart and Yu \cite{ABC3}:
\begin{equation}\label{EqnABC3}
c<  \exp\left( \kappa R^{1/3}(\log R)^3\right)
\end{equation}
where $\kappa$ is an effective constant. See also \cite{ABC1,ABC2,MurtyPasten} for other (effective) bounds, all of which are exponential on a power of $R$, and see \cite{Gyory} for the currently best bounds in the number field setting.

Theorem \ref{ThmMain1} in the archimedian case (thus, only requiring results of Matveev \cite{Matveev}) implies the following effective  bound for the $abc$ conjecture which, in several cases, is subexponential.
\begin{corollary}[Subexponential $abc$ bounds]\label{CoroSubExpABC} Let $\epsilon>0$. There is a number $\kappa_\epsilon>0$ effectively depending on $\epsilon$ such that the following holds:  Let $a$, $b$, $c$ be coprime positive integers with $a+b=c$ and $a<b$. Then
$$
c<   a\cdot \exp\left(\kappa_\epsilon\cdot R^{(1+\epsilon)(\log^*_3 R )/(\log^*_2 R)}\right)
$$
where $R=\rad(abc)$. In particular, if $a< c^{1-\eta}$ for some $\eta>0$, then
$$
c <\exp\left(\eta^{-1}\cdot \kappa_\epsilon\cdot R^{(1+\epsilon)(\log^*_3 R )/(\log^*_2 R)}\right).
$$
\end{corollary}
%%
%%

%%%%
%%%%
%%%%
\subsection{Comparison with other subexponential bounds} For a positive integer $m$ let $\Pcal(m)$ be the largest prime factor of $m$, with $\Pcal(1)=1$.  Under some strong conditions on $\Pcal(a)$, $\Pcal(b)$, $\Pcal(c)$ there are other subexponential bounds for the $abc$ conjecture in the literature.

 In \cite{ABC3}, Stewart and Yu also prove
\begin{equation}\label{EqnP}
c< \exp( p' R^{\kappa\log^*_3 (R)/\log^*_2(R)} )
\end{equation}
where $R=\rad(abc)$, $p'=\min\{\Pcal(a),\Pcal(b), \Pcal(c)\}$ and $\kappa$ is an effective constant. The bound \eqref{EqnP} becomes subexponential on $R$ only when $p'$ is considerably smaller than $R$, namely, when $p'< R^{o(1)}$ with $o(1)$ a function converging to $0$. This last condition is rather restrictive and it depends on the prime factorization of $a$, $b$, and $c$. Over number fields, similar results are given in Theorem 1 of \cite{Gyory08}, proof of Theorem 1 in \cite{Gyory}, and Corollary 7 in \cite{Scoones}.  

 Corollary \ref{CoroSubExpABC} is substantially different in this aspect; it gives a subexponential estimate in a range directly defined by the size of $a$, $b$ and $c$ rather than by their prime factorization. This is relevant since the $abc$ conjecture is precisely an assertion about the prime factorization of $abc$.

After this discussion on the $abc$ conjecture, let us now focus on the topic of Vojta's conjecture with truncated counting functions and our main results.

%%%%
%%%%
%%%%
\subsection{Diophantine approximation functions} All varieties in this article will be assumed to be geometrically irreducible. This is not restrictive: A smooth irreducible variety over $\Q$ having a $\Q$-rational point is necessarily geometrically irreducible.

Let $X$ be a smooth projective variety over $\Q$. We will need some standard functions on rational points of $X$ which play a central role in the study of Diophantine approximation. See \cite{VojtaCIME} for details.

 A bounded function will be denoted by $O(1)$ and, if necessary, dependence on parameters will be indicated as a subscript.

For an invertible sheaf $\Lcal$ on $X$ one has the associated \emph{height} function $h_X(\Lcal,-)$ defined on $X(\Q)$ up adding to a bounded error term.  Given a divisor $D$ on $X$ defined over $\Q$ one can choose a family of \emph{Weil functions} $\lambda_{X,v}(D,-)$ for $v\in M_\Q$, defined on the rational points of $X$ not in $\supp(D)$. The Weil functions are compatible with the height in the following sense: For $x\in X(\Q)$ not in $\supp(D)$ one has 
$$
h_X(\Ocal(D), x) = \sum_{v\in M_\Q} \lambda_{X,v}(D,x) + O(1).
$$
If $D$ is an effective divisor, the \emph{truncated counting function} relative to $D$ is
$$
N^{(1)}_X(D,x)=\sum_{p} \min\{\lambda_{X,p}(D,x), \log p\}
$$
for $x\in X(\Q)$ not in $\supp(D)$, where the sum is over primes $p$. For instance, if $x=b/c\in \Pro^1(\Q)$ with $b,c$ coprime integers, then one has $N^{(1)}_{\Pro^1}(0,x)=\log \rad(b)$ and $N^{(1)}_{\Pro^1}(\infty,x)=\log\rad(c)$ up to a bounded error. Furthermore, if $1-x=a/c$ then $N^{(1)}_{\Pro^1}(1,x)=\log \rad(a)$. Thus, the $abc$ conjecture can be equivalently formulated as the following conjectural bound for $x\in \Pro^1(\Q)-\{0,1,\infty\}$: 
\begin{equation}\label{EqnLogABC}
h_{\Pro^1}(\Ocal(1), x) < (1+\epsilon) N^{(1)}_{\Pro^1}([0]+[1]+[\infty], x) + O_\epsilon(1).
\end{equation}
%%

%%%%
%%%%
%%%%
\subsection{Vojta's conjecture with truncated counting functions}

The canonical sheaf of a variety $X$ will be denoted by $\Kcal_X$. 

In \cite{VojtaABC}, Vojta formulated a general Diophantine approximation conjecture involving truncated counting functions. For rational points over $\Q$ it is the following:
\begin{conjecture}[Vojta's conjecture] \label{ConjVojta} Let $X$ be a smooth projective variety defined over $\Q$. Let $D$ be a reduced effective normal crossings divisor on $X$ defined over $\Q$. Let $\Acal$ be an ample line sheaf on $X$ and let $\epsilon>0$. There is a proper Zariski closed set $Z\subseteq X$ depending on the previous data such that for all $x\in X(\Q)$ not in $Z$ we have
$$
h_X(\Kcal_X\otimes \Ocal(D),x) < N^{(1)}_X(D,x) + \epsilon \cdot h(\Acal,x)+ O(1).
$$
The implicit constant in the error term $O(1)$ can depend on all the data except $x$.
\end{conjecture}
It should be noted that in the special case $X=\Pro^1$ and $D=[0]+[1]+[\infty]$ one recovers the $abc$ conjecture in its formulation \eqref{EqnLogABC}. For varieties $X$ with $X(\Q)$ Zariski dense, the only unconditional results available at this point are for $\Pro^1$; namely, exponential bounds for the $abc$ conjecture from \cite{ABC1,ABC2,ABC3,MurtyPasten}. Our Corollary \ref{CoroSubExpABC} provides a further result in dimension $1$ and, as we will see, Theorem \ref{ThmMain2}  gives an unconditional result for higher dimensional varieties.

%%%%
%%%%
%%%%
\subsection{A special case of Vojta's conjecture} Given a smooth projective variety $X$ over $\Q$, an algebraic point $P\in X(\overline{\Q})$, and a place $v\in M_\Q$, we write $\lambda_{X,v}(P,x)$ for the $v$-adic proximity function to $P$; this is 
$$
\lambda_{X,v}(P,x)=\log^+ (d_{X,v}(P,x)^{-1})
$$
where $\log^+(t)=\log\max\{1,t\}$ and $d_{X,v}(-,-)$ is a $v$-adic distance function on $X(\overline{\Q_v})$, e.g. inherited from a projective embedding. The choice of $d_{X,v}$ only changes $\lambda_{X,v}(P,x)$ by adding a bounded function.

Given an ample line sheaf $\Acal$ on $X$, it is easy to see that there is a constant $M(\Acal,P)>0$ such that for all $x\in X(\Q)-\{P\}$ we have
$$
\lambda_{X,v}(P,x) < M(\Acal,P)\cdot h(\Acal,x) + O(1).
$$
 The optimal value of $M(\Acal,P)$ has been the subject of considerable attention in recent years since the work of McKinnon and Roth \cite{MR} linking it to Seshadri constants.

Thus we arrive at the following consequence of Vojta's Conjecture \ref{ConjVojta}.

\begin{conjecture} \label{ConjMain} Let $X$ be a smooth projective variety defined over $\Q$. Let $D$ be a reduced effective normal crossings divisor on $X$ defined over $\Q$. Let $\Acal=\Kcal_X\otimes \Ocal(D)$ and suppose that $\Acal$ is ample. Let let $P\in X(\overline{\Q})$ be an algebraic point. There is a constant $M$ such that for each $v\in M_\Q$ and every $\epsilon>0$ there is a proper Zariski closed set $Z\subseteq X$ with the following property: For all $x\in X(\Q)$ not in $Z$ we have
$$
\lambda_{X,v}(P,x) <  M\cdot N^{(1)}_X(D,x) + \epsilon \cdot h(\Acal,x)+ O(1)
$$
where the bounded error term $O(1)$ does not depend on $x$.
\end{conjecture}
One can discuss what should be the optimal value of $M$ in the previous conjecture, but we will not pursue that direction here.

%%%%
%%%%
%%%%
\subsection{Diophantine approximation in higher dimensions} 

Our main result is:
\begin{theorem} \label{ThmMain2} Let $X$ be a smooth projective variety over $\Q$ and let $D_1,...,D_m$ be different prime divisors on $X$ defined over $\Q$. Suppose that $D_1,...,D_m$ are linearly dependent modulo linear equivalence over $\Q$. Let $P\in X(\overline{\Q})$ be an algebraic point not in the support of $D=D_1+...+D_m$. There is a proper Zariski closed subset $Z\subseteq X$ such that for every given $\epsilon>0$ and each place $v\in M_\Q$, the following inequality holds  for all $x\in X(\Q)$ not in $Z$:
\begin{equation}\label{EqnMain2}
\lambda_{X,v}(P,x) < (\log^* h_X(\Ocal(D),x))\exp\left(\frac{(1+\epsilon)N^{(1)}_X}{\log^* N^{(1)}_X}\log^*_2 N^{(1)}_X + O(1)\right)
\end{equation}
where the bounded error term $O(1)$ does not depend on $x$, and $N^{(1)}_X= N_X^{(1)}(D,x)$. 
\end{theorem}
We remark that by Hilbert's theorem 90, a divisor defined over $\Q$ is linearly equivalent to $0$ over $\Q$ if and only if it is linearly equivalent to $0$ over $\C$.

\subsection{The linear dependence hypothesis} Naturally, one wants to know when is the linear dependence condition on $D_1,...,D_m$ satisfied in order to be able to apply Theorem \ref{ThmMain2}. It turns out that it is always satisfied provided that $m$ is larger than a constant that only depends on $X$ over $\Q$. To state the result, let us write $\rho(D_1,...,D_m)$ for the rank of the group generated by $D_1,...,D_m$ modulo numerical equivalence and let $\rho_X$ be the Picard rank of $X$. We write $\Pic^0(X/\Q)$ for the group of linear equivalence classes of divisors over $\Q$ which are algebraically equivalent to $0$; this is a finitely generated abelian group (see Paragraph \ref{SecDiv} for details). 
\begin{theorem}\label{ThmRat} Let $X$ be a smooth projective variety over $\Q$ and let $D_1,...,D_m$ be distinct prime divisors on $X$ defined over $\Q$. If
$$
m> \rk \Pic^0(X/\Q) + \rho(D_1,...,D_m)
$$
then the divisors $D_1,...,D_m$ are linearly dependent modulo linear equivalence over $\Q$. In particular, this is the case if $m> \rk \Pic^0(X/\Q)+\rho_X$. 
\end{theorem}
Regarding Theorem \ref{ThmRat}, in applications it might be useful to note that if the irregularity $q(X)=\dim H^0(X,\Omega^1_X)$ vanishes, then $\Pic^0(X/\Q)=(0)$, in which case $m>\rho_X$ suffices.

%%%%
%%%%
%%%%
\subsection{Comparison with Conjecture \ref{ConjMain}} In order to clarify the comparison between our results and Conjecture \ref{ConjMain}, from Theorems \ref{ThmMain2} and \ref{ThmRat} we will deduce:

\begin{corollary}\label{CoroMain} Let $X$ be a smooth projective variety over $\Q$ and let $D_1,...,D_m$ be different prime divisors on $X$ defined over $\Q$. Suppose that $D_1,...,D_m$ are linearly dependent modulo linear equivalence (which is the case, for instance, if $m> \rk \Pic^0(X/\Q) + \rho_X$). Let $P\in X(\overline{\Q})$ be an algebraic point not in the support of $D=D_1+...+D_m$. There is a proper Zariski closed subset $Z\subseteq X$ such that for every given $\epsilon>0$ and every $v\in M_\Q$, the following inequality holds  for all $x\in X(\Q)$ not in $Z$:
\begin{equation}\label{EqnCoroMain}
\lambda_{X,v}(P,x) < \exp\left(\epsilon\cdot  N_X^{(1)}(D,x) \right) +(\log^* h_X(\Ocal(D),x))^{1+\epsilon} + O(1)
\end{equation}
where the bounded error term $O(1)$ does not depend on $x$. 
\end{corollary}

We note that 
$$
(\log^* h_X(\Ocal(D),x))^{1+\epsilon}< \epsilon \cdot h_X(\Ocal(D),x) + O_\epsilon(1)
$$ 
in agreement with Conjecture \ref{ConjMain}.

%%%%
%%%%
%%%%
\subsection{A conjectural improvement}

Lang and Waldschmidt  \cite{Lang, Waldschmidt} have proposed a conjectural improvement on the existing lower bounds for non-vanishing linear forms in logarithms, see Conjecture \ref{ConjLW} for the precise statement. We will show that the Lang-Waldschmidt conjecture implies a version of Conjecture \ref{ConjMain}.

\begin{theorem}\label{ThmLW} Assume the Lang-Waldschmidt conjecture (cf. Conjecture \ref{ConjLW}). Let $X$ be a smooth projective variety over $\Q$ and let $D_1,...,D_m$ be different prime divisors on $X$ defined over $\Q$. Suppose that $D_1,...,D_m$ are linearly dependent modulo linear equivalence (which is the case, for instance, if $m> \rk \Pic^0(X/\Q) + \rho_X$). Let $P\in X(\Q)$ be a rational point not in the support of $D=D_1+...+D_m$. There is a proper Zariski closed subset $Z\subseteq X$ such that for every given $\epsilon>0$  the following inequality holds  for all $x\in X(\Q)$ not in $Z$:
$$
\lambda_{X,\infty}(P,x) < (1+\epsilon)  N_X^{(1)}(D,x)  +\epsilon\cdot h_X(\Ocal(D),x) + O(1)
$$
where the bounded error term $O(1)$ does not depend on $x$. 
\end{theorem}

%%%%
%%%%
%%%%
\subsection{About the proofs} As mentioned before, the main Diophantine approximation input in our arguments comes from the theory of linear forms in logarithms, more precisely, results of  Matveev \cite{Matveev} and Yu \cite{Yu}. From these results we need a consequence due to Evertse and Gy\"ory, namely, Theorem 4.2.1 in \cite{EGbook} (which also needs some new results from geometry of numbers). This will allow us to prove Theorem \ref{ThmMain1} in Section \ref{SecMain1}. An earlier version of this manuscript directly used Matveev's theorem in the archimedian case, and similarly one can use results of Yu in the non-archimedian case, but the theorem of Evertse and Gy\"ory makes the argument go in a smoother way.

Theorem \ref{ThmMain2} is proved in Section \ref{SecMain2}. It is deduced from Theorem \ref{ThmMain1} via a geometric construction.  Theorem \ref{ThmRat} is also proved in Section \ref{SecMain2}. The numerical condition $m>\rk \Pic^0(X/\Q)+\rho(D_1,...,D_m)$ in Theorem \ref{ThmRat}  appears from a variation of a method of Vojta for producing suitable rational maps, see Chapter 2 in \cite{VojtaThesis}. 

Finally, the discussion on the Lang-Waldschmidt conjecture is included in Section \ref{SecLW}.

Let us remark that other applications of linear forms in logarithms to higher dimensional varieties are available in the literature, in the study of integral points. We refer the reader to the work of Levin \cite{Levin} and Le Fourn \cite{leFourn}.

%%%%%%%%%%%%%%%%%%%%%%%%%%%%%%%%%%%%%%
%%%%%%%%%%%%%%%%%%%%%%%%%%%%%%%%%%%%%%
%%%%%%%%%%%%%%%%%%%%%%%%%%%%%%%%%%%%%%
%%%%%%%%%%%%%%%%%%%%%%%%%%%%%%%%%%%%%%
%%%%%%%%%%%%%%%%%%%%%%%%%%%%%%%%%%%%%%
%%%%%%%%%%%%%%%%%%%%%%%%%%%%%%%%%%%%%%

\section{Preliminaries}

%%%
%%%
%%%

\subsection{More on Diophantine approximation functions} A reference for this paragraph is \cite{VojtaCIME}. 

A first fact that we will need is  functoriality of heights, Weil functions, and truncated counting functions.
\begin{lemma} \label{LemmaFunctoriality} Let $f:X\to Y$ be a morphism of smooth projective varieties over $\Q$.
\begin{itemize} 
\item[(i)] Let $\Lcal$ be a line sheaf on $Y$. For all $x\in X(\Q)$ we have
$$
h_X(f^*\Lcal,x)=h_Y(\Lcal, f(x)) + O(1). 
$$
\item[(ii)] Let $D$ be a divisor on $Y$ defined over $\Q$ and let $v\in M_\Q$. Suppose that $f(X)$ is not contained in $\supp(D)$. Then, for all $x\in X(\Q)$ not in $\supp(f^*D)$ we have
$$
\lambda_{X,v}(f^*D,x)=\lambda_{Y,v}(D,f(x)) + O(1).
$$
\item[(iii)] Let $D$ be an effective divisor on $Y$ defined over $\Q$. Suppose that $f(X)$ is not contained in $\supp(D)$. Then, for all $x\in X(\Q)$ not in $\supp(f^*D)$ we have
$$
N^{(1)}_X(f^*D,x)=N^{(1)}_X(D,f(x)) + O(1).
$$
\end{itemize}
\end{lemma} 
We will also need a comparison between the proximity to an algebraic point and the archimedian Weil function relative to a divisor.

\begin{lemma} \label{LemmaProximity} Let $X$ be a smooth projective variety defined over $\Q$, let $D$ be an effective divisor on $X$ defined over $\Q$, and let $P\in X(\overline{\Q})$ be an algebraic point in $\supp(D)$ and let $v\in M_\Q$. Then, for all $x\in X(\Q)$ not in $\supp(D)$ we have
$$
\lambda_{X,v} (P,x)\le \lambda_{X,v}(D,x)+O(1).
$$
\end{lemma}
\begin{proof} The $v$-adic distance from $x$ to $\supp(D)$ is at most the $v$-adic distance from $x$ to $P$, and since $D$ is effective, the $v$-adic Weil function relative to $D$ is at least $-\log d_{X,v}(\supp (D),x) + O(1)$ for $x$ near $\supp(D)$. 
\end{proof}

%%%
%%%
%%%

\subsection{Divisors} \label{SecDiv} For the convenience of the reader, in this section we recall several standard facts about divisors on a smooth projective variety.

Let $Y$ be a smooth projective variety over $\C$. We write $\NS(Y)$ for the group of divisors modulo algebraic equivalence and $\Num(Y)$ for the group of divisors modulo numerical equivalence. We have a surjection $\NS(Y)\to \Num(Y)$. Severi's basis theorem asserts that $\NS(Y)$ is finitely generated, and the kernel of the surjection $\NS(Y)\to \Num(Y)$ is $\NS(Y)_{tor}$, the torsion part of $\NS(Y)$. The Picard rank of $Y$ is $\rho_Y=\rk \NS(Y)=\rk \Num(Y)$.

The group $\Pic(Y/\C)=H^1(Y,\Ocal_Y^\times)$ classifies line sheaves modulo isomorphism, or equivalently, divisors modulo linear equivalence. The group $\Pic^0(Y/\C)$ is defined by the exact sequence
$$
0\to \Pic^0(Y/\C)\to \Pic(Y/\C)\to \NS(Y)\to 0.
$$
That is, $\Pic^0(Y/\C)$ classifies linear equivalence classes of divisors which are algebraically equivalent to $0$. The group $\Pic^0(Y/\C)$ has a natural structure of abelian variety over $\C$ of dimension equal to the irregularity $q(Y)=\dim H^0(Y,\Omega^1_Y)$.

Let $X$ be a smooth projective variety defined over $\Q$. Let us write $\Pic^0(X/\Q)$ for the group of divisors defined over $\Q$ modulo linear equivalence over $\Q$ or over $\C$, which is the same by Hilbert's theorem 90. Let us write $\Pic^0(X/\Q)$ for the subgroup of $\Pic(X/\Q)$ of classes  which are algebraically equivalent to $0$.

Let $Y=X_\C$. There is an abelian variety $P^0_{X/\Q}$ defined over $\Q$, the connected component of the identity of the Picard variety of $X$, satisfying that $\Pic^0(X/\Q)$ is a subgroup of $P^0_{X/\Q}(\Q)$, and $P^0_{X/\Q}(\C)=\Pic^0(Y/\C)$. In particular, the group $\Pic^0(X/\Q)$ is finitely generated by the Mordell-Weil theorem applied to $P^0_{X/\Q}(\Q)$. When $X(\Q)\ne \emptyset$ (or more generally, when $X$ is everywhere locally soluble) we have $P^0_{X/\Q}(\Q)=\Pic^0(X/\Q)$.

%%%
%%%
%%%

\subsection{Linear forms in logarithms and approximation}  Let $h:\overline{\Q}\to \R$ be the absolute logarithmic height, cf. \cite{VojtaCIME}. In the case of $x=b/c\in \Q$ it agrees with the more elementary height $h(x)=\log\max\{|b|,|c|\}$ for $b,c$ coprime integers.

The following result follows from Theorem 4.2.1 in \cite{EGbook}, due to Evertse and Gy\"ory. It follows from bounds for non-vanishing linear forms in logarithms due to Matveev \cite{Matveev} in the archimedian case and Yu \cite{Yu} in the $p$-adic case, combined with some additional results from geometry of numbers. We simply state it in a special case, which is all we need.

\begin{lemma}\label{LemmaEG} Let $\Gamma\subseteq \Q^\times$ be a finitely generated multiplicative group with $q_1,...,q_n\in \Gamma$ generators of $\Gamma/\Gamma_{tor}$. Let $\alpha\in \overline{\Q}^\times$ be an algebraic number and let $d$ be its degree. Let $v\in M_\Q$. There is a number $\kappa_0(\alpha,v)>0$ effectively depending on $\alpha$ and $v$ such that for all $x\in \Gamma$ with $x\ne \alpha$ we have
$$
-\log |\alpha- x|_v < \kappa_0(\alpha,v)\cdot n(\log^*n)(16ed)^{3n} (\log^*h(x))\prod_{j=1}^n h(q_j).
$$ 
\end{lemma}

Theorem 4.2.1 in \cite{EGbook} actually has a term of the form $|1-\beta x|_v$, but choosing $\beta=1/\alpha$ we have $-\log |\alpha- x|_v =-\log|1-\beta x|_v - \log |\alpha|_v$, and $\kappa_0(\alpha,v)$ can absorbe the term  $\log |\alpha|_v$. A precise (and not so complicated) value of $\kappa_0(\alpha,v)$ can be deduced from Theorem 4.2.1 in \cite{EGbook}.

%%%%%%%%%%%%%%%%%%%%%%%%%%%%%%%%%%%%%%
%%%%%%%%%%%%%%%%%%%%%%%%%%%%%%%%%%%%%%
%%%%%%%%%%%%%%%%%%%%%%%%%%%%%%%%%%%%%%
%%%%%%%%%%%%%%%%%%%%%%%%%%%%%%%%%%%%%%
%%%%%%%%%%%%%%%%%%%%%%%%%%%%%%%%%%%%%%
%%%%%%%%%%%%%%%%%%%%%%%%%%%%%%%%%%%%%%

\section{The case of the projective line}\label{SecMain1}

%%%
%%%
%%%

\subsection{Approximations and the radical}

\begin{proof}[Proof of Theorem \ref{ThmMain1}] Let $d$ be the degree of $\alpha$. Let $p_1,...,p_n$ be the prime numbers in the support of $x=b/c$, i.e. primes $p$ with $v_p(x)\ne 0$ where $v_p$ is the $p$-adic valuation. In other words, $p_1,...,p_n$ are the primes dividing $bc$. We may assume $n\ge 1$. Let $\Gamma$ be the multiplicative group generated by $p_1,...,p_n$ and $-1$. Then Lemma \ref{LemmaEG} gives
$$
-\log |\alpha- x|_v <  \kappa_0(\alpha,v)\cdot n(\log^*n)(16ed)^{3n} (\log^*h(x))\prod_{j=1}^n\log p_j.
$$

Using the arithmetic-geometric mean inequality on $\prod_j\log p_j$ and the elementary bound $$n(\log^*n)(16e)^{3n}<e^{12n},$$ we see that
$$
-\log |\alpha- x|_v < \kappa_1(\alpha,v) \cdot \log^*(h(x)) \cdot \exp\left((12+3\log d)n + n\cdot \log\frac{\log \rad(bc)}{n}\right) 
$$
where $\kappa_1(\alpha,v)$ is a number that only depends on $\alpha$ and $v$, in an effective way. 

Let $\epsilon>0$. From analytic number theory (cf. \cite{Robin} for instance) we have that there is an effective $\kappa_2(\epsilon)>0$ depending only on $\epsilon$ such that for all integers $m>\kappa_2(\epsilon)$ we have
$$
\omega(m) < \frac{(1+\epsilon)\log m}{\log\log m}
$$
where $\omega(m)$ is the number of different prime divisors of $m$. Also, note that the function $t\mapsto t\log(A/t)$ is increasing for $t<A/e$. Therefore, applying these observations with $m=\rad(bc)$ and $\omega(m)=n$, we see that there is an effective $\kappa_3(d,\epsilon)$ depending only on $d$ and $\epsilon$ such that
$$
(12+3\log d)n + n\cdot \log\frac{\log \rad(bc)}{n} < \frac{(1+\epsilon)\log \rad(bc)}{\log^*_2 \rad(bc)}\log^*_3\rad(bc) + \kappa_3(d,\epsilon).
$$
We deduce
$$
-\log |\alpha- x|_v < (\log^*h(x)) \cdot \exp\left(\frac{(1+\epsilon)\log \rad(bc)}{\log^*_2 \rad(bc)}\log^*_3\rad(bc) + \kappa_4(v,\alpha,\epsilon)\right)
$$
where $\kappa_4(v,\alpha,\epsilon)$ only depends on $v$, $\alpha$ and $\epsilon$, in an effective way.
\end{proof}

%%%
%%%
%%%

\subsection{Reformulation on $\Pro^1$} It will be convenient to reformulate Theorem \ref{ThmMain1} as an approximation statement on $\Pro^1$ relative to the divisor $D=[0]+[\infty]$. 

\begin{theorem}[Reformulation of Theorem \ref{ThmMain1}] \label{ThmMain1bis} Let $\alpha\in \Pro^1(\overline{\Q})-\{0,\infty\}$, let $v\in M_\Q$ and let $\epsilon>0$. There is a number $\kappa(v,\alpha,\epsilon)>0$ depending effectively on $v$, $\alpha$ and $\epsilon$ such that for all $x\in\Pro^1(\Q)-\{0,\alpha,\infty\}$ we have
$$
\lambda_{\Pro^1,v}(\alpha,x)< (\log^* h(x)) \exp\left(\frac{(1+\epsilon)N^{(1)}_{\Pro^1}}{\log^*N^{(1)}_{\Pro^1}}\log^*_2 N^{(1)}_{\Pro^1} + \kappa(v,\alpha,\epsilon)\right)
$$
where $N^{(1)}_{\Pro^1}=N^{(1)}_{\Pro^1}([0]+[\infty],x)$.
\end{theorem}

%%%
%%%
%%%

\subsection{Consequences for $abc$}

\begin{proof}[Proof of Corollary \ref{CoroSubExpABC}] We apply Theorem \ref{ThmMain1}  with $v=\infty$ the archimedian place, $\alpha=1$ and $x=b/c$. One observes that $-\log |1-x|=\log(c/a)$ and then we use \eqref{EqnABC3} (or any of \cite{ABC1,ABC2,ABC3,MurtyPasten}) to get bound 
$$
\log^*_2h(x) = \log^*_3 c +O(1) < \log^*_2 R + O(1)
$$ 
with an effective error term. In this way we get
$$
\log c -\log a < \exp\left(\log^*_2 R + \frac{(1+\epsilon)\log \rad(bc)}{\log^*_2 \rad(bc)}\log^*_3 \rad(bc)  +O_\epsilon(1)\right)
$$
where the implicit constant effectively depends on $\epsilon$. The result follows.
\end{proof}

%%%%%%%%%%%%%%%%%%%%%%%%%%%%%%%%%%%%%%
%%%%%%%%%%%%%%%%%%%%%%%%%%%%%%%%%%%%%%
%%%%%%%%%%%%%%%%%%%%%%%%%%%%%%%%%%%%%%
%%%%%%%%%%%%%%%%%%%%%%%%%%%%%%%%%%%%%%
%%%%%%%%%%%%%%%%%%%%%%%%%%%%%%%%%%%%%%
%%%%%%%%%%%%%%%%%%%%%%%%%%%%%%%%%%%%%%

\section{The main result in arbitrary dimension}\label{SecMain2}

%%%
%%%
%%%

\subsection{The Main Theorem} 

\begin{proof}[Proof of Theorem \ref{ThmMain2}] \label{SecMainProof} Since $D_1,...,D_m$ are linearly dependent modulo linear equivalence over $\Q$, there is a non-constant rational function $f:X\dasharrow \Pro^1$ defined over $\Q$ such that the effective divisor $f^*([0]+[\infty])$ is supported on $D$. Let $E$ be the locus where $f$ is not defined; it has codimension at least $2$ in $X$. 

We claim that $E$ is contained in $\supp(D)$. Indeed, let $H_0=f^*[0]$ and $H_\infty = f^*[\infty]$; these are linearly equivalent divisors on $X$ supported on $D$. Then $E$ is the base locus of the linear system determined by $f$, which is contained in $\supp(H_0)\cap\supp(H_\infty)\subseteq \supp(D)$.

There is a smooth projective variety $\widetilde{X}$ over $\Q$ with morphisms $\pi:\widetilde{X}\to X$ and $\widetilde{f}:\widetilde{X}\to \Pro^1$ defined over $\Q$ such that the following holds:
\begin{itemize}
\item[(i)] $\pi$ is an isomorphism above $X-E$, and
\item[(ii)] $f\circ \pi= \widetilde{f}$ as rational functions on $\widetilde{X}$.
\end{itemize}
That is, the triple $(\widetilde{X},\pi,\widetilde{f})$ is a resolution of the rational map $f$. Let $F_0=\widetilde{f}^*[0]$ and $F_\infty=\widetilde{f}^*[\infty]$; these are effective non-zero divisors on $\widetilde{X}$. Notice that $\widetilde{f}^*\Ocal(2)\simeq \Ocal(F_0+F_\infty)$, hence, Lemma \ref{LemmaFunctoriality} shows that for all $x'\in \widetilde{X}(\Q)$

\begin{equation}\label{Eqn1}
2h(\widetilde{f}(x')) = h_{\widetilde{X}}(\Ocal(F_0+F_\infty),x')+O(1).
\end{equation}

Choose $M$ a large enough positive integer such that $M\cdot \pi^*D\ge F_0+F_\infty$, which is possible since $f^*([0]+[\infty])=H_0+H_\infty$ is supported on $D$ and $E\subseteq \supp(D)$.

Let $P\in X(\overline{\Q})$ be an algebraic point not in $ \supp(D)$ (hence, not in $E$) and let $P'\in \widetilde{X}(\overline{\Q})$ be its only preimage by $\pi$. Let $\alpha=f(P)=\widetilde{f}(P')\in \Pro^1(\overline{\Q})$ and notice that $\alpha\ne 0,\infty$ since $f^*([0]+[\infty])$ is supported on $D$ and $E\subseteq \supp(D)$. Let $A$ be the divisor on $\Pro^1$  constructed from the Galois orbit of $\alpha$; it is defined over $\Q$ and it has $\alpha$ in its support. Let $G=\widetilde{f}^*A$, which is a divisor on $\widetilde{X}$ defined over $\Q$ with $P'$ in its support, and define 
$$
Z=\supp(D)\cup \pi(\supp(G))
$$
which is a proper Zariski closed subset of $X$.

Fix $v\in M_\Q$. For every $x\in (X-Z)(\Q)$ there is a unique point $x'\in \widetilde{X}(\Q)$ with $\pi(x')=x$, and for such points the following holds (for simplicity, expressions are written up to adding a bounded error term $O(1)$ independent of $x$ and $x'$, and (ii) and Lemma \ref{LemmaFunctoriality} are used several times without mentioning them):
\begin{itemize}
\item[(a)] By  \eqref{Eqn1} and $F_0+F_\infty\le M\cdot \pi^*D$ we have 
$$
\begin{aligned}
2h(f(x))&=2h(\widetilde{f}(x')) = h_{\widetilde{X}}(\Ocal(F_0+F_\infty),x')\\
&\le h_{\widetilde{X}}(\Ocal(M\cdot \pi^*D), x')= M\cdot h_X(\Ocal(D),x).
\end{aligned}
$$ 
\item[(b)]  By Lemma \ref{LemmaProximity} and (i) we have
$$
\lambda_{X,v}(P,x) = \lambda_{\widetilde{X},v}(P',x') \le \lambda_{\widetilde{X},v}(G,x') = \lambda_{\Pro^1,v}(A, f(x)).
$$ 
\item[(c)] Let $\alpha_1,...,\alpha_h\in \Pro^1(\overline{\Q})$ be the Galois conjugates of $\alpha$. Since $f(x)$ can be close to at most one of them in the $v$-adic metric, we have  
$$
\lambda_{\Pro^1,v}(A, f(x))=\max_{1\le j\le h}\lambda_{\Pro^1,v}(\alpha_j, f(x)).
$$
\item[(d)] Since $F_0+F_\infty\le M\cdot \pi^*D$ and the truncated counting function only depends on the support of an effective divisor, we find
$$
\begin{aligned}
N^{(1)}_{\Pro^1}([0]+[\infty],f(x))&=N^{(1)}_{\widetilde{X}}(F_0+F_\infty,x')\\
&\le N^{(1)}_{\widetilde{X}}(\pi^*D,x') =N^{(1)}_{X}(D,x).
\end{aligned}
$$
\end{itemize}
Let $\epsilon>0$ and, in the notation of Theorem \ref{ThmMain1bis}, let $\kappa'(v,\alpha,\epsilon)=\max_{1\le j\le h} \kappa(v,\alpha_j,\epsilon)$. Define
$$
\xi_\epsilon(t)=\exp\left(\frac{(1+\epsilon)t}{\log^*t}\log^*_2 t + \kappa'(v,\alpha,\epsilon)\right).
$$
Then, for $x\in (X-Z)(\Q)$ we have
$$
\begin{aligned}
\lambda_{X,v}(P,x) & <  \max_{1\le j\le h}\lambda_{\Pro^1,v}(\alpha_j, f(x)) + O(1) &\mbox{ by (b) and (c)} \\
& <  (\log^* h(f(x)))\xi_\epsilon\left(N^{(1)}_{\Pro^1}([0]+[\infty],f(x))\right) +O(1)& \mbox{ by Theorem \ref{ThmMain1bis}}\\
& <  \left(\log^* \left(\frac{M}{2}\cdot h(\Ocal(D),x)+O(1)\right)\right)\xi_\epsilon\left(N^{(1)}_{X}(D,x)+O(1)\right) +O(1)& \mbox{ by (a) and (d).}
\end{aligned}
$$
We conclude by adjusting the number $\kappa'(v,\alpha,\epsilon)$.
\end{proof}
%%
%%

%%%
%%%
%%%

\subsection{Linear dependence}

\begin{proof}[Proof of Theorem \ref{ThmRat}] Let $r=\rk \Pic^0(X/\Q)$ and $n=\rho(D_1,...,D_m)$, so that $m>r+n$.  Let $\Gamma$ be the group generated by the linear equivalence classes of $D_1,...,D_m$ over $\Q$ or, equivalently, over $\C$. Then $\rk \Gamma \le r+n$ by $\ker(\NS(X_\C)\to \Num(X_\C))=\NS(X_\C)_{tor}$, by the exact sequence
$$
0\to \Pic^0(X_\C/\C)\to \Pic(X_\C/\C)\to \NS(X_\C)\to 0,
$$
and by the fact that $\Gamma\le \Pic(X/\Q)$. See Paragraph \ref{SecDiv} for details. Since $m>r+n$, the divisors $D_1,...,D_m$ cannot be linearly independent modulo linear equivalence over $\Q$.
\end{proof}
%%
%%

%%%
%%%
%%%

\subsection{Combining Theorems \ref{ThmMain2} and \ref{ThmRat}} 

\begin{proof}[Proof of Corollary \ref{CoroMain}] Notice that for any $\epsilon>0$ and positive real numbers $A,B$, we have $AB < A^{1+\epsilon} + B^{1+1/\epsilon}$. Using this inequality twice, we see that \eqref{EqnMain2} implies \eqref{EqnCoroMain}. By Theorem \ref{ThmRat}, the necessary linear dependence condition follows from the given numerical hypothesis on $m$.
\end{proof}

%%%%%%%%%%%%%%%%%%%%%%%%%%%%%%%%%%%%%%
%%%%%%%%%%%%%%%%%%%%%%%%%%%%%%%%%%%%%%
%%%%%%%%%%%%%%%%%%%%%%%%%%%%%%%%%%%%%%
%%%%%%%%%%%%%%%%%%%%%%%%%%%%%%%%%%%%%%
%%%%%%%%%%%%%%%%%%%%%%%%%%%%%%%%%%%%%%
%%%%%%%%%%%%%%%%%%%%%%%%%%%%%%%%%%%%%%

\section{Remarks on the Lang-Waldschmidt Conjecture}\label{SecLW}

%%%
%%%
%%%

\subsection{The Lang-Waldschmidt Conjecture} Let us recall the following conjecture proposed by Lang and Waldschmidt, which predicts a strong archimedian lower bound for non-vanishing linear forms in logarithms. We state it in its multiplicative form. See p.212-217 in \cite{Lang} and Conjecture 2.5 in \cite{Waldschmidt}:
\begin{conjecture}[Lang-Waldschmidt 1978] \label{ConjLW} Let $\epsilon>0$. There is a number $C(\epsilon)$ depending only on $\epsilon$ such that for all positive integers $a_1,...,a_n$ and non-zero integers $b_1,...,b_n$ with $a_1^{b_1}\cdots a_n^{b_n}\ne 1$ we have 
$$
\left|a_1^{b_1}\cdots a_n^{b_n} - 1\right| \ge \frac{C(\epsilon)\max_j |b_j|}{\left|b_1\cdots b_na_1\cdots a_n\right|^{1+\epsilon}}.
$$
\end{conjecture}
A heuristic for this conjecture is provided in p.212-217 of \cite{Lang}.

%%%
%%%
%%%

\subsection{A consequence} 

\begin{lemma}\label{LemmaLW} Suppose that the Lang-Waldschmidt conjecture holds. Let $\epsilon>0$. For all $x\in \Pro^1-\{0,1,\infty\}$ we have
$$
\lambda_{\Pro^1,\infty}(1,x) < (1+\epsilon)N^{(1)}([0]+[\infty],x) + \epsilon h(x) + O_\epsilon(1).
$$
\end{lemma}
\begin{proof} We let $p_1,...,p_n$ be the primes in the support of $x$; we may assume $n\ge 1$ and $x>0$. Let us write $x=p_1^{b_1}\cdots p_n^{b_n}$ with integer non-zero exponents $b_j$. The Lang-Waldschmidt conjecture with the bound $\max_j |b_j|\ge 1$ gives
$$
\lambda_{\Pro^1,\infty}(1,x) < n\cdot c(\epsilon) + (1+\epsilon)\log \prod_{j=1}^n|b_j| + (1+\epsilon)\log Q
$$
with $c(\epsilon)=-\log C(\epsilon)$ and $Q=\prod_{j=1}^np_j$. By elementary number theory (e.g. using a bound for the number of divisors of an integer) for each $\delta>0$ we have 
$$
\log \prod_{j=1}^n|b_j| < \delta \log (p_1^{|b_1|}\cdots p_n^{|b_n|}) + O_\delta(1) \le 2\delta h(x) + O_\delta(1).
$$
We notice that for every $\eta>0$ we have
$$
n=\omega(Q)< \eta \log Q + O_\eta(1),
$$
and
$$
 \log Q = N^{(1)} ([0]+[\infty],x)+ O(1).
$$
Putting these bounds together and adjusting $\epsilon, \delta, \eta$ we get the result.
\end{proof}
%%%
%%%
%%%

\subsection{The conjectures of Lang-Waldschmidt and Vojta}

\begin{proof}[Proof of Theorem \ref{ThmLW}]  There is a rational function $f:X\dasharrow \Pro^1$ with $f^*([0]+[\infty])$ supported on $D$. The indeterminacy locus of $f$ is contained in $\supp(D)$ and we may multiply $f$ by a suitable rational number to achieve $f(P)=1$. After this point, the proof goes in exactly the same way as the proof of Theorem \ref{ThmMain2} (cf. Section \ref{SecMainProof}), using Lemma \ref{LemmaLW} instead of Theorem \ref{ThmMain1}.
\end{proof}

%%%%%%%%%%%%%%%%%%%%%%%%%%%%%%%%%%%%%%
%%%%%%%%%%%%%%%%%%%%%%%%%%%%%%%%%%%%%%
%%%%%%%%%%%%%%%%%%%%%%%%%%%%%%%%%%%%%%
%%%%%%%%%%%%%%%%%%%%%%%%%%%%%%%%%%%%%%
%%%%%%%%%%%%%%%%%%%%%%%%%%%%%%%%%%%%%%
%%%%%%%%%%%%%%%%%%%%%%%%%%%%%%%%%%%%%%

\section{Acknowledgments}

This research was supported by ANID (ex CONICYT) FONDECYT Regular grant 1190442 from Chile.

I heartily thank K\'alm\'an Gy\"ory for valuable feedback on a first version of this manuscript and for suggesting to use Theorem 4.2.1 \cite{EGbook} instead of \cite{Matveev, Yu}, which led to a more streamlined presentation.

%%%%%%%%%%%%%%%%%%%%%%%%%%%%%%%%%%%%%%
%%%%%%%%%%%%%%%%%%%%%%%%%%%%%%%%%%%%%%
%%%%%%%%%%%%%%%%%%%%%%%%%%%%%%%%%%%%%%

\end{document}